\documentclass[12pt]{article}

\usepackage[margin=4cm]{geometry}

\usepackage[dvipsnames]{xcolor}

\usepackage[utf8]{inputenc}
\usepackage[T1]{fontenc}

\usepackage{amsmath, amssymb, amsthm}

\usepackage{stmaryrd}

\usepackage[numbers]{natbib}

\usepackage{tikz, tikz-cd}

\usepackage{graphicx}

\usepackage{listings}

\usepackage{proof}

\usepackage{bbm, bm}

\usepackage{hyperref}

\usepackage[textwidth=3.7cm]{todonotes}
\usepackage{parskip}

\usepackage{fancyhdr}

\usepackage[framemethod=TikZ]{mdframed}

\usepackage{quiver}

\makeatletter
\def\moverlay{\mathpalette\mov@rlay}
\def\mov@rlay#1#2{\leavevmode\vtop{%
   \baselineskip\z@skip \lineskiplimit-\maxdimen
   \ialign{\hfil$\m@th#1##$\hfil\cr#2\crcr}
}}
\newcommand{\charfusion}[3][\mathord]{
    #1{\ifx#1\mathop\vphantom{#2}\fi
        \mathpalette\mov@rlay{#2\cr#3}
      }
    \ifx#1\mathop\expandafter\displaylimits\fi}
\makeatother

\newtheorem{thm}{Theorem}
\newtheorem*{thm*}{Theorem}

\newtheorem*{rmk}{Remark}

\newtheorem{cor}{Corollary}
\newtheorem*{cor*}{Corollary}

\theoremstyle{definition}
\newtheorem{defn/}{Definition}
\newtheorem*{defn*/}{Definition}

\newcommand{\defnendsymbol}%
{%
  \mathbin{\rotatebox[origin=c]{-45}{$\parallel$}}%
}

\newenvironment{defn*}
  {%
   \pushQED{\qed}\begin{defn*/}}
  {\popQED\end{defn*/}}

\theoremstyle{theorem}

\newcommand*{\catFont}[1]{\mathsf{#1}} 
\newcommand*{\catVarFont}[1]{\mathcal{#1}}

\newcommand{\Grp}{\catFont{Grp}}

\newcommand{\catC}{\catVarFont{C}}
\newcommand{\catD}{\catVarFont{D}}

\DeclareMathOperator{\dif}{d \!}

\newcommand*{\pd}[3][]{\ensuremath{
\ifinner
\tfrac{\partial{^{#1}}#2}{ \partial{#3^{#1}} }
\else
\dfrac{\partial{^{#1}}#2}{ \partial{#3^{#1}} }
\fi
}}

\newcommand*{\od}[3][]{\ensuremath{
\ifinner
\tfrac{\dif{^{#1}}#2}{ \dif{#3^{#1}} }
\else
\dfrac{\dif{^{#1}}#2}{ \dif{#3^{#1}} }
\fi
}}

\newcommand*{\interior}[1]{ {\kern0pt#1}^{\mathrm{o}} }

\newbox\gnBoxA
\newdimen\gnCornerHgt
\setbox\gnBoxA=\hbox{$\ulcorner$}
\global\gnCornerHgt=\ht\gnBoxA
\newdimen\gnArgHgt
\def\godelnum #1{%
\setbox\gnBoxA=\hbox{$#1$}%
\gnArgHgt=\ht\gnBoxA%
\ifnum     \gnArgHgt<\gnCornerHgt \gnArgHgt=0pt%
\else \advance \gnArgHgt by -\gnCornerHgt%
\fi \raise\gnArgHgt\hbox{$\ulcorner$} \box\gnBoxA %
\raise\gnArgHgt\hbox{$\urcorner$}}

\newcommand{\teq}{\triangleq}

\newcommand*{\important}[1]{\textcolor{MidnightBlue}{\emph{#1}}}

\IfFileExists{../preamble.tex}{\input{../preamble.tex}}{}

\newcommand{\Gph}{\mathsf{Gph}}

\author{Chris Grossack\\ (they/them)}
\title{The Right Angled Artin Group Functor as a Categorical Embedding}

\begin{document}
\maketitle

\begin{abstract}
  It has long been known that the combinatorial properties of a graph $\Gamma$
  are closely related to the group theoretic properties of its 
  \emph{right angled artin group} (raag). It's natural to ask if the 
  graph \emph{homomorphisms} are similarly related to the group homomorphisms 
  between two raags. The main result of this paper shows that there 
  is a purely algebraic way to characterize the raags amongst groups, and 
  the graph homomorphisms amongst the group homomorphisms. 
  As a corollary we 
  present a new algorithm for recovering $\Gamma$ from its raag.
\end{abstract}

\section{Introduction}
\label{intro}
  For us, a \important{graph} $\Gamma$ with underlying vertex set $V$ is a 
  symmetric, reflexive relation on $V$. A \important{graph homomorphism} from 
  a graph $(V,\Gamma)$ to $(W,\Delta)$ is a 
  function $\varphi : V \to W$ so that $(v_1, v_2) \in \Gamma \implies (\varphi v_1, \varphi v_2) \in \Delta$.
  These assemble into a category, which we call $\mathsf{Gph}$.

  Given a graph $\Gamma$ with vertex set $V$, we can form a group $A\Gamma$, the 
  \important{right angled artin group} (raag) associated to $\Gamma$, defined as

  \[ A\Gamma \teq \langle v \in V \mid [v_1, v_2] = 1 \text{ whenever } (v_1,v_2) \in \Gamma \rangle .\]

  For example, if $K_n$ is a complete graph on $n$ vertices then 
  $AK_n \cong \mathbb{Z}^n$. If $\Delta_n$ is a discrete graph on $n$ vertices
  $A\Delta_n \cong \mathbb{F}_n$ is a free group on $n$ generators. 
  If $\square$ is the graph with $4$ vertices $a,b,c,d$ and 
  four edges $(a,b)$, $(b,c)$, $(c,d)$, and $(d,a)$ then $A \square \cong 
  \langle a, c \rangle \times \langle b, d \rangle \cong \mathbb{F}_2 \times \mathbb{F}_2$.
  In this sense, raags allow us to \emph{interpolate} between free and free
  abelian groups. 

  Raags are of particular interest to geometric group theorists because
  of their connections to the fundamental groups of closed hyperbolic
  $3$-manifolds \cite{servatiusSurfaceSubgroupsGraph1989} 
  and to the mapping class groups of hyperbolic surfaces 
  \cite{kimGeometryCurveGraph2014}. Moreover, raags were instrumental in
  the resolution of the Virtual Haken Conjecture \cite{agolVirtualHakenConjecture2013}
  due to their close connection with the CAT(0) geometry of cube complexes. 
  See \cite{bestvinaGeometricGroupTheory2013} for an overview.

  Importantly, the combinatorial structure of $\Gamma$ is closely related to the 
  algebraic structure of $A \Gamma$, with useful information flow in both directions. 
  For instance, the cohomology of $A \Gamma$ is the \emph{exterior face algebra} 
  of $\Gamma$ \cite{sabalkaRigidityIsomorphismProblem2009},
  $A\Gamma$ factors as a direct product if and only if $\Gamma$ factors
  as a join of two graphs \cite{servatiusAutomorphismsGraphGroups1989},
  and we can compute the Bieri-Neumann-Strebel invariant $\Sigma^1(A\Gamma)$
  from just information in $\Gamma$ \cite{meierBieriNeumannStrebelInvariantsGraph1995}.
  This correspondence can be pushed remarkably far, and recently it was
  shown that \emph{expander graphs}\footnote{which are really sequences of graphs}
  can be recognized from the cohomology of their raags
  \cite{floresExpandersRightangledArtin2021}! For more information about the 
  close connection between the combinatorics of $\Gamma$ and the algebra of
  $A\Gamma$, see 
  \cite{floresExpandersRightangledArtin2021, koberdaGeometryCombinatoricsRightAngled2022}.
  
  With this context, it is natural to ask whether the combinatorics
  of graph homomorphisms are \emph{also} closely connected to the algebra of
  group homomorphisms between raags. For a particular example, one might ask 
  if there is a purely algebraic way to recognize when a group homomorphism 
  between raags is $A \varphi$ for some homomorphism $\varphi$ of their 
  underlying graphs.

  The main result of this paper shows
  that the answer is \emph{yes} in a very strong sense. We prove that the raag 
  functor $A$ is an equivalence between the category of graphs $\mathsf{Gph}$ 
  and the category of groups equipped with a 
  coalgebra structure\footnote{A kind of \important{descent data}}
  that we will describe shortly. As corollaries, we obtain a new way of
  recognizing the raags amongst the groups, and the graph homomorphisms 
  amongst the group homomorphisms. This moreover gives a new algorithm for 
  recovering the underlying graph of a raag from nothing but its isomorphism 
  type.

  Crucial for the proof of this theorem is the fact that 
  $A : \mathsf{Gph} \to \Grp$ has a right adjoint, the \important{commutation graph}
  functor $C : \Grp \to \mathsf{Gph}$ which sends a group $G$ to the graph 
  whose vertices are elements of $G$ and where $(g_1,g_2) \in CG \iff [g_1,g_2]=1$. 
  This is surely well known to 
  experts\footnote{It's implicit in the ``universal property of raags'' given 
  in \cite{koberdaRightangledArtinGroups2012}, for instance, and is stated as
  such in \cite{servatiusAutomorphismsGraphGroups1989}} but is not often 
  mentioned in the literature. This is likely because of the common convention that graphs 
  have no self loops, whereas the adjunction requires us to work with graphs
  with a self loop at each vertex. Of course, this does not appreciably 
  change the combinatorics, and we feel it is a small price to pay for the 
  categorical clarity this adjunction provides.

  Unsurprisingly, the commutation graph and related constructions have already 
  been of interest to combinatorialists for many years 
  \cite{baumeisterCommutingGraphsOdd2009,
  dolinarMaximalDistancesCommuting2012,
  giudiciDiametersCommutingGraphs2010,
  arvindRecognizingCommutingGraph2022,
  cameronGraphsDefinedGroups2022}, and the complement of the commutation 
  graph was even the subject of a (now proven) conjecture of Erd\H{o}s
  \cite{neumannProblemPaulErdos1976}.

  With the commutation graph functor $C$ defined, we can state the main result
  of this paper:

  \begin{thm*}
    The right angled artin group functor $A : \mathsf{Gph} \to \mathsf{Grp}$ is comonadic. 

    That is, $A$ is an equivalence of categories between $\Gph$ and 
    the category $\mathsf{Grp}_{AC}$ of groups equipped with an $AC$-coalgebra structure,
    and group homomorphisms that are moreover $AC$-cohomomorphisms. 
  \end{thm*}

  The group $ACG$ is freely generated by symbols $[g]$ for each $g \in G$, 
  with relations saying $[g][h] = [h][g]$ in $ACG$ if and only if $gh = hg$ in $G$. 
  Write $\epsilon_G : ACG \to G$ for the map sending each $[g] \mapsto g$.
  Additionally, write $\delta : ACG \to AC(ACG)$ for the map sending each
  $[g] \mapsto [[g]]$.

  Now merely unwinding the category theoretic definitions gives the following corollary:

  \begin{cor*}[Main Corollary]
    An abstract group $G$ is isomorphic to a raag if and only if it admits a 
    group homomorphism $\mathfrak{g} : G \to ACG$ so that the following two 
    diagrams commute:

    \[
        \begin{tikzcd}
        G & ACG && G   & ACG \\
          & G   && ACG & AC(ACG)
        \arrow["\mathfrak{g}", from=1-1, to=1-2]
        \arrow["{\epsilon_G}", from=1-2, to=2-2]
        \arrow["{1_G}"', from=1-1, to=2-2]
        \arrow["\mathfrak{g}"', from=1-4, to=2-4]
        \arrow["\mathfrak{g}", from=1-4, to=1-5]
        \arrow["\delta"', from=2-4, to=2-5]
        \arrow["AC\mathfrak{g}", from=1-5, to=2-5]
        \end{tikzcd}
    \]

    Moreover, a group homomorphism $f : G \to H$ between raags is 
    $A \varphi$ for some graph homomorphism $\varphi$ if and only if it 
    respects these structure maps in the sense that 

    \[
        \begin{tikzcd}
        G & H \\
        ACG & ACH
        \arrow["\mathfrak{g}", from=1-1, to=2-1]
        \arrow["\mathfrak{h}", from=1-2, to=2-2]
        \arrow["f", from=1-1, to=1-2]
        \arrow["ACf"', from=2-1, to=2-2]
        \end{tikzcd}
    \]    

    commutes.
  \end{cor*}

  \begin{rmk}
    In particular, there is a purely algebraic way to recognize the 
    raags amongst the groups and the image of the graph homomorphisms 
    amongst the group homomorphisms between raags.

    This additionally gives us a new way to recover $\Gamma$ from the 
    abstract isomorphism class of $A \Gamma$, and shows it is 
    decidable (even efficient!) to check whether any particular
    group homomorphism between raags came from a graph homomorphism.
  \end{rmk}

  Our proof uses some category theory that might not be familiar to all
  readers, so in Section \ref{review} we will briefly review the machinery 
  of \important{comonadic descent}, which is the main technical tool for 
  the proof (which is the subject of Section \ref{proof}). 
  First, though, in Section \ref{eg} we give an example to show that 
  category theory is not needed in order to apply our results.
  This section might also be of interest to 
  those learning category theory looking for toy examples of comonadic descent, 
  since it is usually applied in more complicated situations than this%
  \footnote{Indeed, this is how the author came upon this result.}. 
  Lastly, in Section \ref{computing} we discuss the algorithmic consequences 
  of the main result.

  Throughout this paper, we make the notational convention that graph theoretic 
  concepts are written with greek letters and group theoretic concepts with 
  roman letters. The coalgebraic structure maps are written in fraktur font.

\section{An Instructive Example}
\label{eg}

It's important to note that applying this result requires no knowledge of 
the deep category theory used in its proof. Let's begin with a simple example
of how the result can be used to detect whether a group 
homomorphism came from a graph homomorphism or not.

Let $\Gamma = \{ v \}$ and $\Delta = \{ w \}$ be two one-vertex graphs.
Then $A \Gamma = \langle v \rangle$ and $A \Delta = \langle w \rangle$,
and we want to detect when a homomorphism between these groups came from a
homomorphism of their underlying graphs.

Recall that $CG$, the commutation graph of $G$, has a vertex $[g]$ for 
each $g \in G$, with an edge relating $[g]$ and $[h]$ exactly when 
$g$ and $h$ commute in $G$. So $C \langle v \rangle$ is a complete 
graph on $\mathbb{Z}$ many vertices labelled by $[v^n]$.

Then the group $ACG$ is freely generated by the symbols $[g]$, for $g \in G$,
subject to relations saying $[g][h] = [h][g]$ in $ACG$ if and only if 
$gh = hg$ in $G$. So $AC \langle v \rangle$ is the free abelan group with 
generators $[v^n]$.

It's not hard to see that the map 
$\mathfrak{v} : \langle v \rangle \to AC \langle v \rangle$
sending $v \mapsto [v^1]$ satisfies the axioms from the Main Corollary. 
The existence of such a $\mathfrak{v}$ tells us that 
$\langle v \rangle$ must be a raag, which it is%
\footnote{More generally, if we know $\Gamma$, then the map 
$A\Gamma \to AC(A\Gamma)$ sending each generator $\gamma \mapsto [\gamma]$ 
will always satisfy the axioms.}.

\bigskip

Let's first look at a map that \emph{does} come from a graph homomorphism,
for instance $f : \langle v \rangle \to \langle w \rangle$ given by $fv = w$.

The corollary says to consult the following square:

\[
  \begin{tikzcd}
    {\langle v \rangle} && {\langle w \rangle} \\
    \\
    {\left \langle [v^n] \mid [v^n][v^m] = [v^m][v^n] \right \rangle} 
    && 
    {\left \langle [w^n] \mid [w^n][w^m] = [w^m][w^n] \right \rangle}
    \arrow["{v \mapsto w}", from=1-1, to=1-3]
    \arrow["{[v^n] \mapsto [w^n]}", from=3-1, to=3-3]
    \arrow["{v \mapsto [v^1]}"{description}, from=1-1, to=3-1]
    \arrow["{w \mapsto [w^1]}"{description}, from=1-3, to=3-3]
  \end{tikzcd}
\]

and since this is quickly seen to commute, we learn that $f$ is of the form 
$A \varphi$ for some graph homomorphism (as indeed it is).

\bigskip

Next, let's look at a map which \emph{doesn't} come from a graph homomorphism,
like $f : \langle v \rangle \to \langle w \rangle$ given by $fv = w^2$.

Now our square is

\[\begin{tikzcd}
	{\langle v \rangle} && {\langle w \rangle} \\
	\\
	{\langle v^n \mid [v^n,v^m] = 1 \rangle} && {\langle w^n \mid [w^n, w^m] = 1 \rangle}
	\arrow["{v \mapsto w^2}", from=1-1, to=1-3]
  \arrow["{[v^n] \mapsto [w^{2n}]}", from=3-1, to=3-3]
	\arrow["{v \mapsto [v^1]}"{description}, from=1-1, to=3-1]
	\arrow["{w \mapsto [w^1]}"{description}, from=1-3, to=3-3]
\end{tikzcd}\]

which does \emph{not} commute (even though it seems to at first glance). 
Indeed, if we chase the image of $v$ around the top right of the square, then 
we see

\[ v \mapsto w^2 \mapsto [w^1]^2 \]

If instead we chase around the lower left of the square, we get:

\[ v \mapsto [v^1] \mapsto [w^2] \]

since $[w^1]^2 \neq [w^2]$ in this group (recall $AC \langle w \rangle$ is 
freely generated by the symbols $[w^n]$), we have successfully detected that 
$f$ did \emph{not} come from a graph homomorphism!

Importantly, this same approach works even if we merely know the 
coalgebra structures on $G$ and $H$. Thus we don't need to know their
underlying graphs to detect the graph homomorphisms%
\footnote{Though we will see later that the coalgebra structure 
actually lets us recover the underlying graphs as well.}! 

As a last aside, let's mention what the structure map 
$\mathfrak{g} : G \to ACG$ does. Elements of $ACG$ are 
formal words in the elements of $G$. Then, intuitively, 
$\mathfrak{g}(g) = [\gamma_1][\gamma_2]\cdots[\gamma_k]$ decomposes 
$g$ as a formal product of the vertices making up $g$. This means 
that we can recover the vertices of $\Gamma$ as those $g$ 
so that $\mathfrak{g}(g) = [g]$ is a word of length $1$, as we prove 
in Section \ref{computing}

\section{A Brief Review of Comonadic Descent}
\label{review}

Recall that an \important{adjunction}
$(L : \catC \to \catD) \dashv (R : \catD \to \catC)$ is a 
pair of functors equipped with a natural isomorphism

\[ \text{Hom}_\mathcal{D}(LC,D) \cong \text{Hom}_\mathcal{C}(C,RD). \]

Of particular interest for us is the adjunction $A \dashv C$ specifying the 
universal property of raags.

Recall moreover that a \important{comonad} $W : \catD \to \catD$ is a 
functor equipped with natural transformations $\epsilon : W \Rightarrow 1_\catD$ 
and $\delta : W \Rightarrow WW$ so that the following diagrams of 
natural transformations commute:

\[
    \begin{tikzcd}
    & W &&& W & WW \\
    W & WW & W && WW & WWW
    \arrow["{1_W \cdot \epsilon}", Rightarrow, from=2-2, to=2-1]
    \arrow["{\epsilon \cdot 1_W}"', Rightarrow, from=2-2, to=2-3]
    \arrow["\delta", Rightarrow, from=1-2, to=2-2]
    \arrow["{1_W}"', Rightarrow, from=1-2, to=2-1]
    \arrow["{1_W}", Rightarrow, from=1-2, to=2-3]
    \arrow["\delta"', Rightarrow, from=1-5, to=2-5]
    \arrow["\delta", Rightarrow, from=1-5, to=1-6]
    \arrow["{\delta \cdot 1_W}"', Rightarrow, from=2-5, to=2-6]
    \arrow["{1_W \cdot \delta}", Rightarrow, from=1-6, to=2-6]
    \end{tikzcd}
\]

Dually, a \important{monad} is a functor $M : \catC \to \catC$ equipped with 
natural transformations $\eta : 1_\catC \Rightarrow M$ and 
$\mu : MM \Rightarrow M$ satisfying diagrams opposite those above. A precise
definition can be found in Chapter $4$ of \cite{borceuxCategoriesStructures1994}.

Every adjunction $L \dashv R$ gives rise to a monad $RL$ and a comonad 
$LR$. In particular, the raag adjunction gives us a comonad 
$AC : \Grp \to \Grp$, which is our primary object of study.

Monads and comonads find application in settings as varied as algebraic 
geometry and number theory
\cite{grothendieckTechniqueDescenteTheoremes1958, 
borceuxCategoriesStructures1994,
borceuxMonadicApproachGalois2010},
universal algebra 
\cite{borceuxCategoriesStructures1994,
adamekAlgebraicTheoriesCategorical2011,
bojanczykRecognisableLanguagesMonads2015,
awodeyCoalgebraicDualBirkhoff2000,
hylandCategoryTheoreticUnderstanding2007},
probability theory 
\cite{giryCategoricalApproachProbability1982,
culbertsonCategoricalFoundationBayesian2014,
leinsterCODENSITYULTRAFILTERMONAD2013},
and computer science 
\cite{moggiNotionsComputationMonads1991,
depaivaDialecticaComonads2021,
ghaniAlgebrasCoalgebrasMonads2001,
ruttenUniversalCoalgebraTheory2000}.
Relevant for us is the theory of \important{(co)monadic descent}, 
which comes from gluing conditions in algebraic geometry, 
and is reviewed in this context in Section $4.7$ of 
\cite{borceuxCategoriesStructures1994}.

Given a comonad $W$, a \important{$W$-coalgebra} is an object
$G \in \catD$ equipped with an arrow $\mathfrak{g} : G \to WG$ so that the
diagrams in Figure \ref{fig:coalg} commute. A \important{$W$-cohomomorphism}
between coalgebras $(G,\mathfrak{g})$ and $(H,\mathfrak{h})$ is an arrow $f : G \to H$
in $\mathcal{D}$ compatible with $\mathfrak{g}$ and $\mathfrak{h}$, in the sense that 
Figure \ref{fig:cohom} commutes. When $W$ is clear from context, we simply call these
coalgebras and cohomomorphisms, and they assemble into a category $\catD_W$
which admits a faithful functor $U : \catD_W \to \catD$ that 
simply forgets the structure map $\mathfrak{g}$.

\begin{figure}
    \[
        \begin{tikzcd}
        G & WG && G & WG \\
        & G && WG & WWG
        \arrow["\mathfrak{g}", from=1-1, to=1-2]
        \arrow["{\epsilon_G}", from=1-2, to=2-2]
        \arrow["{1_G}"', from=1-1, to=2-2]
        \arrow["\mathfrak{g}"', from=1-4, to=2-4]
        \arrow["\mathfrak{g}", from=1-4, to=1-5]
        \arrow["\delta"', from=2-4, to=2-5]
        \arrow["W\mathfrak{g}", from=1-5, to=2-5]
        \end{tikzcd}
    \]
    \caption{The defining diagrams for a coalgebra}
    \label{fig:coalg}
\end{figure}

\begin{figure}
    \[
        \begin{tikzcd}
        G & H \\
        WG & WH
        \arrow["\mathfrak{g}", from=1-1, to=2-1]
        \arrow["\mathfrak{h}", from=1-2, to=2-2]
        \arrow["f", from=1-1, to=1-2]
        \arrow["Wf"', from=2-1, to=2-2]
        \end{tikzcd}
    \]    
    \caption{The defining diagram for a cohomomorphism}
    \label{fig:cohom}
\end{figure}

Abstract nonsense shows that for any adjunction $L \dashv R$, the 
essential image of $L$ lands inside the category of coalgebras $\mathcal{D}_{LR}$.
That is, every object $LX \in \mathcal{D}$ is a $LR$-coalgebra, 
where the structure map is given by $L \eta_X : LX \to LRLX$, and every 
$L \varphi$ is a $LR$-cohomomorphism.

In our special case, $\eta : \Gamma \to CA\Gamma$ is the map sending 
each $v \in \Gamma$ to $v^1 \in CA\Gamma$. Then the above says that the 
functor $A : \mathsf{Gph} \to \mathsf{Grp}$ factors through the 
category of coalgebras $\mathsf{Grp}_{AC}$ as follows:

\[\begin{tikzcd}
	{\mathsf{Gph}} & {\mathsf{Grp}_{AC}} & {\mathsf{Grp}} \\
	\Gamma & {(A\Gamma, A \eta)} & A\Gamma
	\arrow[maps to, from=2-1, to=2-2]
	\arrow[maps to, from=2-2, to=2-3]
	\arrow["A", from=1-1, to=1-2]
	\arrow["U", from=1-2, to=1-3]
\end{tikzcd}\]

We will show that $A$ is actually an equivalence of categories 
$\mathsf{Gph} \simeq \Grp_{AC}$. This tells us that a group is of the 
form $A\Gamma$ if and only if it's a coalgebra, and a group homomorphism is 
of the form $A\varphi$ if and only if it's a cohomomorphism!

The main tool for proving this equivalence is Beck's famed
\important{(Co)Monadicity Theorem}\footnote{The original manuscript due to Beck 
was unpublished, but widely distributed. A scan is available at
\cite{beckBeckMonadicityTheorem1968}, but this is also proven as
Theorem 4.4.4 in \cite{borceuxCategoriesStructures1994}. Both of 
these prove the statement for \emph{monads}, which is then dualized to give
the \emph{co}monadicity theorem we use.}, which says

\begin{thm*}[Beck, 1968]
  To show that a left adjoint
  $(L : \catC \to \catD) \dashv (R : \catD \to \catC)$ 
  witnesses $L$ as an equivalence of categories $\catC \simeq \catD_{LR}$%
  \footnote{Such an adjunction $L \dashv R$ is called \emph{comonadic}.},
  it suffices to show

  \begin{enumerate}
      \item $L$ reflects isomorphisms (that is, whenever 
        $L \varphi : L\Gamma \cong L\Delta$ is an isomorphism in $\catD$, 
        then $\varphi$ must have already been an isomorphism in $\catC$)
      \item $\catC$ has, and $L$ preserves, equalizers of coreflexive pairs%
        \footnote{We will recall the definition of a coreflexive pair in 
        section \ref{proof}}
  \end{enumerate}
\end{thm*}

This gives us our outline for proving the main theorem:

\begin{thm*}[Main Theorem]
  The right angled artin group functor $A : \Gph \to \Grp$ restricts to 
  an equivalence of categories $A : \Gph \simeq \Grp_{AC}$ between the 
  category of graphs and the full subcategory of groups equipped with an 
  $AC$-coalgebra structure.
\end{thm*}

\begin{proof}
  By Beck's comonadicity theorem, it suffices to check the two conditions 
  above.

  Condition $(1)$ is a classical result due to Droms 
  \cite{dromsIsomorphismsGraphGroups1987}, 
  so it remains to check $(2)$. It's well known that $\mathsf{Gph}$
  is complete%
  \footnote{one quick way to see this is to note that it's \emph{topologically concrete}
  in the sense of \cite{adamekAbstractConcreteCategories2009}}, and thus has all equalizers. 

  In the next section we'll recall the definition of a coreflexive pair, 
  and show that $A$ really does preserve their equalizers. This will 
  complete the proof.
\end{proof}

\section{The Raag Functor Preserves Equalizers of Coreflective Pairs}
\label{proof}

A \important{coreflexive pair} is a pair of arrows with a common retract. 
That is, a diagram

\[\begin{tikzcd}
	\Gamma && \Delta
	\arrow["\alpha", shift left=3, from=1-1, to=1-3]
	\arrow["\beta"', shift right=3, from=1-1, to=1-3]
	\arrow["\rho"{description}, from=1-3, to=1-1]
\end{tikzcd}\]

where $\rho \alpha = 1_\Gamma = \rho \beta$.

Now, we want to show that if $\Theta$ is the equalizer of $\alpha$ and $\beta$,
as computed in $\mathsf{Gph}$, 
then $A\Theta$ should still be the equalizer of $A \alpha$ and $A \beta$, 
as computed in $\mathsf{Grp}$.
For ease of notation, we will confuse $\alpha$ and $\beta$ with $A \alpha$
and $A \beta$, since 
$(A\alpha)(v_1^{n_1} v_2^{n_2} \cdots v_k^{n_k}) = (\alpha v_1)^{n_1} (\alpha v_2)^{n_2} \cdots (\alpha v_k)^{n_k}$.

Now, $\Theta$ is quickly seen to be the full subgraph of $\Gamma$
on the vertices where $\alpha v = \beta v$. 
So then $A\Theta = \langle v \mid \alpha v = \beta v \rangle \leq A \Gamma$.
If instead we compute the equalizer of $A \alpha$ and $A \beta$ in $\Grp$,
we get $G = \{ g \mid \alpha g = \beta g \} \leq A\Gamma$.

So showing that $A \Theta = G$ amounts to showing that, provided $\alpha$ 
and $\beta$ admit a common retract $\rho$, each $g$ 
with $\alpha g = \beta g$ is a word in those vertices $v$ with $\alpha v = \beta v$.

\begin{thm}
    The right angled artin group functor $A$ preserves equalizers of coreflexive pairs
\end{thm}

\begin{proof}
    Since $\rho$ is a graph
    homomorphism, we see that $v$ and $w$ are $\Gamma$-related if and only if $\alpha v$
    and $\alpha w$ (equivalently $\beta v$ and $\beta w$, equivalently $\alpha v$ and $\beta w$) 
    are $\Delta$-related. Thus $v$ and $w$ commute in $A \Gamma$ if and only 
    if their images under $\alpha$ and $\beta$ commute in $A \Delta$.

    In Theorem 3.9 of her thesis \cite{greenGraphProductsGroups1990}, 
    Green proves that elements of $A \Gamma$ have a normal form as words 
    in the vertices of $\Gamma$\footnote{In fact, she proves something slightly more general}.
    Following the exposition of Koberda \cite{koberdaRightangledArtinGroups2012} 
    and others, we call a word $w \in A \Gamma$ \important{central} if the 
    letters in $w$ pairwise commute. This happens if and only if the letters 
    in $w$ form a clique in $\Gamma$. We say that $w$ is in 
    \important{central form} if it is a product of central words 
    $w = w_1 w_2 \cdots w_k$. If we stipulate that we are ``left greedy'' 
    in the sense that no letter in $w_{i+1}$ commutes with each letter of $w_i$%
    \footnote{so that we first make $w_1$ as long as possible, then make $w_2$ as long as possible, and so on},
    then the central form is unique up to commuting the letters in each $w_i$.
    See also Section 3.3 of \cite{charneyIntroductionRightangledArtin2007} for a summary.

    Now suppose that $\alpha g = \beta g$. Fix such a central form
    $g = w_0 w_1 \ldots w_k$, and look at 

    \[ (\alpha w_0) (\alpha w_1) \ldots (\alpha w_k) = (\beta w_0) (\beta w_1) \ldots (\beta w_k) \]

    these representations of $\alpha g = \beta g$ are both minimal length, as 
    we could hit a shorter representation with $\rho$ in order to get a 
    shorter representation for $g$. Then uniqueness of the central form
    says that each $\alpha w_i$ and $\beta w_i$ are equal up to permuting the 
    letters in each. 
    
    We restrict attention to each 
    $w_i = \gamma_1^{n_1} \gamma_2^{n_2} \ldots \gamma_k^{n_k}$ separately, say

    \[ 
    (\alpha \gamma_1^{n_1}) (\alpha \gamma_2^{n_2}) \ldots (\alpha \gamma_k^{n_k}) = 
    \delta_1^{n_1} \delta_2^{n_2} \ldots \delta_k^{n_k} =
    (\beta \gamma_1^{n_1}) (\beta \gamma_2^{n_2}) \ldots (\beta \gamma_k^{n_k})
    \]

    If we can show that actually $\alpha \gamma_i = \beta \gamma_i$ for each $i$,
    then we'll be done.

    But $\alpha$ and $\beta$ give injections from $\{ \gamma_1 \ldots, \gamma_k \}$
    to $\{ \delta_1, \ldots, \delta_k \}$, which are in fact bijections since we're
    dealing with finite sets of the same cardinality. 

    Moreover, by assumption $\rho$ provides an inverse for $\alpha$ \emph{and} 
    for $\beta$!

    Then $\alpha$ and $\beta$ must be the same map on this set, and in particular
    each $\gamma_i$ satisfies $\alpha \gamma_i = \beta \gamma_i$, as desired.
\end{proof}

\section{Can we Really Compute These?}
\label{computing}

It is well known that the problem 
``is a finitely presented group $G$ isomorphic to a raag'' 
is undecidable. Indeed, being isomorphic to a raag is a 
\emph{Markov property} in the sense of Definition 3.1 in \cite{millerDecisionProblemsGroups1992}
so Theorem 3.3 in the same paper guarantees this problem is undecidable.

Let's work with the next best thing, then, and suppose we're given a 
finitely presented group $G$ and a promise that it \emph{is} a raag 
(though we are not given its underlying graph). How much can we learn 
about the combinatorics of its underlying graph from just $G$?

First, we must find an $AC$-coalgebra structure on $G$ -- that is, a group 
homomorphism $\mathfrak{g} : G \to ACG$ satisfying the conditions from 
Figure \ref{fig:coalg}. Since $ACG$ is a raag, it has solvable word problem, so 
we can enumerate all homomorphisms $G \to ACG$ and check if they satisfy the 
axioms. We will eventually find such a $\mathfrak{g}$ since we were 
promised that $G$ is abstractly isomorphic to a raag, so this algorithm 
terminates. Recall also that that if we happen to already know the underlying 
graph that we have an explicit formula for the coalgebra structure. The unique 
map sending each generator $\gamma \in A \Gamma$ to $[\gamma] \in ACA\Gamma$ 
always works.

Once we know the coalgebra structures on $G$ and $H$, we can already 
efficiently check whether a group homomorphism $f : G \to H$ came from a 
graph homomorphism.

\begin{thm}
  Given a homomorphism $f : G \to H$ between finitely presented groups%
  \footnote{Recall that these presentations may have nothing to do with the 
  underlying graphs} where $(G,\mathfrak{g})$ and $(H,\mathfrak{h})$ are 
  moreover $AC$-coalgebras, then there is an algorithm deciding 
  whether $f$ is $A \varphi$ for $\varphi$ a graph homomorphism of the graphs 
  presenting $G$ and $H$.
\end{thm}

\begin{proof}
    By the equivalence $\mathsf{Gph} \simeq \Grp_{AC}$, this amounts to 
    checking if $f$ is a cohomomorphism -- that is, whether the square

    \[
        \begin{tikzcd}
            G & H \\
            ACG & ACH
            \arrow["f", from=1-1, to=1-2]
            \arrow["\mathfrak{g}"', from=1-1, to=2-1]
            \arrow["ACf"', from=2-1, to=2-2]
            \arrow["\mathfrak{h}", from=1-2, to=2-2]
        \end{tikzcd}
    \]

    commutes. Of course, we can check this on the (finitely many) generators
    of $G$, 
    and the claim now follows from the fact that $ACH$ is a raag%
    \footnote{We have to be a bit careful, since $CH$ is infinite, so that 
    $ACH$ is not finitely generated. However, the images of each generator 
    of $G$ will land in a finite subgraph of $CH$, so we can do our 
    computation inside the raag associated to that finite subgraph.}, and thus
    has solvable word problem \cite{charneyIntroductionRightangledArtin2007}.
\end{proof}

\begin{cor}
  There is an algorithm to recover $\Gamma$ from the mere isomorphism 
  type $G$ of $A \Gamma$.
\end{cor}

\begin{proof}
  We know that the vertices of $\Gamma$ are in bijection with 
  graph homomorphisms from the one-vertex graph $1$ to $\Gamma$. 
  By the equivalence $\mathsf{Gph} \simeq \Grp_{AC}$, this amounts 
  to cohomomorphisms $\mathbb{Z} \to G$, which one can explicitly 
  calculate to be those elements $g \in G$ so that $\mathfrak{g}(g) = [g]$.

  Since we know that the number of vertices of $\Gamma$ is equal to the 
  rank of the abelianization $G^\text{ab}$, we can keep checking elements 
  of $G$ to see if $\mathfrak{g}(g) = [g]$. This algorithm terminates 
  because once we've found $\text{rk}(G^\text{ab})$ many such elements, 
  we must have found all of them.

  Finally, we see that the following conditions are equivalent:
  \begin{enumerate}
    \item Two elements $g_1, g_2$ represent adjacent elements in $\Gamma$ 
    \item $g_1$ and $g_2$ commute in $G$
    \item $[g_1]$ and $[g_2]$ commute in $ACG$
    \item There is a cohomomorphism from $A(\bullet - \bullet)$ to $G$ 
      sending the two vertices to $g_1$ and $g_2$
  \end{enumerate}
  
\end{proof}

\section{Conclusion}
\label{conclusion}

It has been well known for some time now that the combinatorics of a graph
$\Gamma$ are reflected in the algebra of its raag $A\Gamma$, but the 
question of how the combinatorics of graph homomorphisms relates to 
group homomorphisms between raags remains fertile ground. In this paper we've 
shown that the connection remains strong, by showing that the category of
(reflexive) graphs embeds faithfully as an explicit subcategory of 
the category of groups. 

More speculatively, while this paper focused on the comonad $AC : \Grp \to \Grp$,
we suspect there is a future role to be played by the monad 
$CA : \mathsf{Gph} \to \mathsf{Gph}$. Indeed, Kim and Koberda conjecture 
in \cite{kimEmbedabilityRightangledArtin2013} that embeddings 
$A\Gamma \to A \Delta$ exist exactly when $\Gamma$ embeds into a 
graph $\Delta^e$ which they call the \emph{extension graph}.
This graph is closely related to the monad graph $CA\Delta$ 
(indeed, it's the full subgraph of $CA\Delta$ on the conjugates of generators), 
as we might expect since maps $A\Gamma \to A \Delta$ are in natural 
bijection with maps $\Gamma \to CA \Delta$. 

While the extension graph conjecture is now known to be false in general
\cite{casals-ruizEmbedddingsPartiallyCommutative2013}, it is true for 
many classes of graphs. In some sense this is likely ``because of'' the close 
connection of the extension graph with the monad graph. It would be interesting 
to see if category theoretic techniques can be brought to bear on a new version 
of this conjecture, by finding a combinatorial condition which picks out those 
embeddings $\Gamma \to CA \Delta$ which transpose to an embedding of raags.

\section*{Acknowledgements}

The author would like to thank Matt Durham, Jacob Garcia, Thomas Koberda, and 
Peter Samuelson for helpful conversations and encouragement during the 
writing of this paper. Additionally, the author would like to thank 
Carl-Fredrik Nyberg-Brodda for a preliminary reading, and informing them 
about the Casals-Ruiz paper falsifying the extension graph conjecture.

\newpage
\bibliographystyle{plain}
\bibliography{descent-raags.bib}

\end{document}